\documentclass[12pt]{amsart}%
\usepackage{amsmath}
\usepackage{amsfonts}
\usepackage{amssymb}
\usepackage{graphicx,delarray}
\usepackage{graphicx}%
\providecommand{\U}[1]{\protect\rule{.1in}{.1in}}
\theoremstyle{plain}
\newtheorem{theorem}{Theorem}
\newtheorem*{theorem*}{Theorem}

\newtheorem{corollary}{Corollary}

\numberwithin{equation}{section}
\allowdisplaybreaks
\begin{document}

\title{Several explicit formulae of sums and hyper-sums of powers of integers}

\author{Fouad Bounebirat }
\address[F. Bounebirat]{USTHB, RECITS Laboratory, Faculty of Mathematics, P.O. Box 32, El Alia,16111, Algiers, Algeria.}
\email{bounebiratfouad@yahoo.fr}
\author{Diffalah Laissaoui }
\address[D. Laissaoui]{Faculty of Science, University Yahia Far\`{e}s M\'{e}d\'{e}a, urban pole,
	26000 M\'{e}d\'{e}a, Algeria.}
\email{laissaoui.diffalah74@gmail.com}
\author{Mourad Rahmani}
\address[M. Rahmani]{USTHB, Faculty of Mathematics, P. O. Box 32, El Alia,16111, Algiers, Algeria.}
\email{mrahmani@usthb.dz}

\begin{abstract}
In this paper, we present several explicit formulas of the sums and hyper-sums of the powers of the first $\left(  n+1\right)  $-terms of a general arithmetic sequence in
terms of Stirling numbers and generalized Bernoulli polynomials. 
\end{abstract}
\maketitle

\section{Introduction}

The problem of finding formulas for sums of powers of integers has attracted
the attention of many mathematicians and has been developed in several
different directions. For a recent treatment and references, see \cite{Bazso,Beardon, Edwards, Laiss2, Merca2014, Mezo2014b}. This paper is concerned both with sums $S_{p,\left(
	a,d\right)  }\left(  n\right)  $ and hyper-sums $S_{p,\left(  a,d\right)
	\text{ }}^{\left(  r\right)  }\left(  n\right)  $ of the $p$-the powers of the
first $\left(  n+1\right)  $-terms of a general arithmetic sequence. \ Let%

\begin{align*}
S_{p,\left(  a,d\right)  }\left(  n\right)   &  =a^{p}+\left(  a+d\right)
^{p}+\cdots+\left(  a+nd\right)  ^{p}\\
&  =%
{\displaystyle\sum\limits_{s=0}^{n}}
\left(  a+sd\right)  ^{p},
\end{align*}
be the power sum of arithmetic progression with $n,p$ are non-negative
integers and $a$ and $d$ are complex numbers with $d\neq0.$ For the most
studied case $a=0$ and $d=1$%

\[
S_{p,\left(  0,1\right)  }\left(  n\right)  =\left\{
\begin{array}
[c]{c}%
n+1\text{ \ \ \ \ \ \ \ \ \ \ \ \ \ \ \ \ \ \ \ \ \ \ \ \ \ \ }(p=0)\\
1^{p}+2^{p}+3^{p}+\cdots+n^{p}\text{ \ \ \ \ }\left(  p>0\right)  
\end{array}
\right.,
\]
there have been a considerable number of results.

The basic properties for the $S_{p,\left(  a,d\right)  }\left(  n\right)  $
can be obtained from the following generating function \cite{Howard1996}
\begin{equation}%
{\displaystyle\sum\limits_{p\geq0}}
S_{p,\left(  a,d\right)  }\left(  n\right)  \frac{z^{p}}{p!}=%
{\displaystyle\sum\limits_{k=0}^{n}}
e^{\left(  a+kd\right)  z}, \label{L1}%
\end{equation}
and we can easily verify that \cite{Merca2014}
\begin{equation}
S_{p,\left(  a,d\right)  }\left(  n\right)  =\frac{d^{p}}{p+1}\left(
B_{p+1}\left(  n+\frac{a}{d}+1\right)  -B_{p+1}\left(  \frac{a}{d}\right)
\right)  , \label{BBr}%
\end{equation}
where $B_{n}\left(  x\right)  $ denotes the classical Bernoulli polynomials,
which are defined by the following generating function%
\[
\frac{ze^{xz}}{e^{z}-1}=%
{\displaystyle\sum\limits_{n\geq0}}
B_{n}\left(  x\right)  \frac{z^{n}}{n!}.
\]



Recall that the weighted Stirling numbers $\mathcal{S}_{n}^{i}\left(  x\right)
$ of the second kind are defined by (see \cite{Carlitz1, Carlitz2})
\begin{align}
\mathcal{S}_{n}^{i}\left(  x\right)   &  =\frac{1}{i!}\Delta^{i}x^{n}\\
&  =\frac{1}{i!}{\sum\limits_{j=0}^{i}}\left(  -1\right)  ^{i-j}\dbinom{i}%
{j}\left(  x+j\right)^{n}, \label{GS1}
\end{align}
where $\Delta$ denotes the forward difference operator. The
exponential generating function of $\mathcal{S}_{n}^{k}\left(
x\right)  $ is given by
\begin{equation}
{\sum\limits_{n=i}^{\infty}}\mathcal{S}_{n}^{i}\left(  x\right)  \frac{z^{n}%
}{n!}=\frac{1}{i!}e^{xz}\left(  e^{z}-1\right)  ^{i} \label{gensti}
\end{equation}
and $\mathcal{S}_{n}^{i}\left(  x\right)
$ satisfy the following recurrence relation:
\[
\mathcal{S}_{n+1}^{i}\left(  x\right)  =\mathcal{S}_{n}^{i-1}\left(
x\right) +\left(  x+i\right)  \mathcal{S}_{n}^{i}\left(  x\right)
\qquad(1\leq i\leq n).
\]
In particular, we have for nonnegative integer $r$
\[
\mathcal{S}_{n}^{i}\left(  0\right)  =\genfrac{\{}{\}}{0pt}{0}{n}{i}\ \ \text{and} \ \ \mathcal{S}_{n}^{i}\left(  r\right)  =%
\genfrac{\{}{\}}{0pt}{0}{n+r}{i+r}_{r},
\]
where $\genfrac{\{}{\}}{0pt}{0}{n}{i}_{r}$ denotes the $r$-Stirling
numbers of the second kind \cite{Broder}. These numbers counts the number of partitions of a set of $n$ objects into exactly $k$ nonempty, disjoint subsets, such that the first $r$ elements are in distinct subsets.

For any positive integer $m$. The $r$-Whitney numbers of the second kind $W_{m,r}(n,i)$ are the coefficients in the expansion
\[
\left(  mx+r\right)  ^{n}=%
{\displaystyle\sum\limits_{i=0}^{n}}
m^{i}W_{m,r}(n,i)x(x+1)\cdots(x+i-1),
\]
and given by their generating function%
\[
{\sum\limits_{n\geq i}}W_{m,r}(n,i)\frac{z^{n}}{n!}=\frac{1}{m^{i}i!}%
e^{rz}\left(  e^{mz}-1\right)  ^{i}.
\]
Clearly, we have%
\[
W_{1,0}(n,i)=%
\genfrac{\{}{\}}{0pt}{}{n}{i}, W_{1,r}(n,i)=%
\genfrac{\{}{\}}{0pt}{}{n+r}{i+r}%
_{r}
\]
and
\[
W_{m,r}\left(  n,i\right) =m^{n-i}\mathcal{S}_{n}^{i}\left(  \frac{r}{m}\right).
\]
For more details of these numbers see \cite{Rahmani2014b}.


\section{The sums of powers of integers $S_{p,\left(  a,d\right)  }\left(
	n\right)  $}

An explicit formula for $S_{p,\left(  a,d\right)  }\left(  n\right)  $ is
given by the following:

\begin{theorem}
	\label{Principal}For all integers $n,p\geq0$ and $a$,$d$ are complex numbers
	with $d\neq0,$ we have%
	\[
	S_{p,\left(  a,d\right)  }\left(  n\right)  =d^{p}{\sum\limits_{k=0}^{p}%
	}k!\dbinom{n+1}{k+1}\mathcal{S}_{p}^{k}\left(  \frac{a}{d}\right)  .
	\]
	
\end{theorem}

\begin{proof}
	It follows from (\ref{gensti}) that%
	\begin{align*}%
	{\displaystyle\sum\limits_{p\geq0}}
	\left(  d^{p}{\sum\limits_{k=0}^{p}}k!\dbinom{n+1}{k+1}\mathcal{S}_{p}%
	^{k}\left(  \frac{a}{d}\right)  \right)  \frac{z^{p}}{p!}  &  ={\sum
		\limits_{k\geq0}}k!\dbinom{n+1}{k+1}%
	{\displaystyle\sum\limits_{p\geq0}}
		\mathcal{S}_{p}^{k}\left(  \frac{a}{d}\right)  \frac{\left(  dz\right)  ^{p}%
	}{p!}\\
	&  =e^{az}{\sum\limits_{k\geq0}}\dbinom{n+1}{k+1}\left(  e^{dz}-1\right)
	^{k}\\
	&  =e^{az}\frac{e^{\left(  n+1\right)  dz}-1}{e^{dz}-1}\\
	&  =%
		{\displaystyle\sum\limits_{k=0}^{n}}
		e^{\left(  a+kd\right)  z}%
	\end{align*}
	and the proof is complete.
\end{proof}

The following Corollary immediately follows from Theorem \ref{Principal}.

\begin{corollary}
	If we assume that $d$ divides $a$, then we have for $p>0$
	\[
	S_{p,\left(  a,d\right)  }\left(  n\right)  =d^{p}%
	{\displaystyle\sum\limits_{k=0}^{p}}
		k!\dbinom{n+1}{k+1}%
		\genfrac{\{}{\}}{0pt}{0}{p+\frac{a}{d}}{k+\frac{a}{d}}%
		_{\frac{a}{d}}.
	\]
	
\end{corollary}

The next Corollary contains an explicit formula for $S_{p,\left(  a,d\right)
}\left(  n\right)  $ expressed in terms of the $r$-Whitney numbers of the
second kind $W_{m,r}(n,k).$

\begin{corollary}
	If we assume that $a$ and $d$ are coprime integers, then we have for $p\geq0$
	\[
	S_{p,\left(  a,d\right)  }\left(  n\right)  =%
		{\displaystyle\sum\limits_{k=0}^{p}}
		k!d^{k}\dbinom{n+1}{k+1}W_{a,d}\left(  p,k\right)  .
	\]
	
\end{corollary}

An explicit formula for $S_{p,\left(  a,d\right)  }\left(  n\right)  $
involving Bernoulli polynomials is given by the following Theorem.

\begin{theorem}%
	\[
	S_{p,\left(  a,d\right)  }\left(  n\right)  =\frac{d^{p}}{p+1}%
		{\displaystyle\sum\limits_{s=0}^{p}}
		\dbinom{p+1}{s}\left(  n+1\right)  ^{p+1-s}B_{s}\left(  \frac{a}{d}\right)  ,
	\]
	
\end{theorem}

\begin{proof}
	It follows from \cite{Rahmani2015} that%
	\begin{equation}
	B_{n}\left(  x\right)  =%
		{\displaystyle\sum\limits_{k=0}^{n}}
		\left(  -1\right)  ^{k}\frac{k!}{k+1}\mathcal{S}_{n}^{k}\left(  x\right)  .
	\label{RAHMANI}%
	\end{equation}
	Thus (\ref{BBr}) becomes%
	\[
	S_{p,\left(  a,d\right)  }\left(  n\right)  =\frac{d^{p}}{p+1}%
		{\displaystyle\sum\limits_{k=0}^{p+1}}
		\left(  -1\right)  ^{k}\frac{k!}{k+1}\left(  \mathcal{S}_{p+1}^{k}\left(  n+\frac{a}%
	{d}+1\right)  -\mathcal{S}_{p+1}^{k}\left(  n+\frac{a}{d}+1\right)  \right)  .
	\]
	Now, from $\left(  \ref{GS1}\right)$, we get%
	\begin{align*}
	S_{p,\left(  a,d\right)  }\left(  n\right)   &  =\frac{d^{p}}{p+1}%
		{\displaystyle\sum\limits_{k=0}^{p+1}}
		\left(  -1\right)  ^{k}\frac{k!}{k+1}\left(  \frac{1}{k!}%
		{\displaystyle\sum\limits_{j=0}^{k}}
		\left(  -1\right)  ^{k-j}\dbinom{k}{j}\left[
		{\displaystyle\sum\limits_{s=0}^{p}}
		\dbinom{p+1}{s}\left(  \frac{a}{d}+j\right)  ^{s}\left(  n+1\right)
	^{p+1-s}\right]  \right) \\
	&  =%
		{\displaystyle\sum\limits_{s=0}^{p}}
		\dbinom{p+1}{s}\left(  n+1\right)  ^{p+1-s}\frac{1}{k!}%
		{\displaystyle\sum\limits_{j=0}^{k}}
		\left(  -1\right)  ^{k-j}\dbinom{k}{j}\left(  \frac{a}{d}+j\right)  ^{s}\\
	&  =\frac{d^{p}}{p+1}%
		{\displaystyle\sum\limits_{k=0}^{p+1}}
	\left(  -1\right)  ^{k}\frac{k!}{k+1}\left(
		{\displaystyle\sum\limits_{s=0}^{p}}
		\dbinom{p+1}{s}\left(  n+1\right)  ^{p+1-s}\mathcal{S}_{s}^{k}\left(  \frac{a}%
	{d}\right)  \right) \\
	&  =\frac{d^{p}}{p+1}%
	{\displaystyle\sum\limits_{s=0}^{p}}
		\dbinom{p+1}{s}\left(  n+1\right)  ^{p+1-s}\left(
		{\displaystyle\sum\limits_{k=0}^{p+1}}
		\left(  -1\right)  ^{k}\frac{k!}{k+1}\mathcal{S}_{s}^{k}\left(  \frac{a}{d}\right)
	\right)  .
	\end{align*}
	Using again $\left(  \ref{RAHMANI}\right)$, we get the desired result.
\end{proof}

\section{The hyper-sums of powers of integers $S_{p,\left(
		a,d\right)  }^{\left(  r\right)  }\left(  n\right)  $}

The hyper-sums of powers of integers $S_{p,\left(  a,d\right)  }^{\left(
	r\right)  }\left(  n\right)  $ $(p\geq0)$ (or the $r$-fold summation of $p$th
powers) are defined recursively as%
\[
\left\{
\begin{array}
[c]{c}%
S_{p,\left(  a,d\right)  }^{\left(  0\right)  }\left(  n\right)
={\sum\limits_{i=0}^{n}}\left(  a+id\right)  ^{p}\text{ }\\
S_{p,\left(  a,d\right)  \text{ }}^{\left(  r\right)  }\left(  n\right)
={\sum\limits_{j=0}^{n}}S_{p,\left(  a,d\right)  }^{\left(  r-1\right)
}\left(  j\right)
\end{array}
\right.  .
\]

In this section, we generalize the results obtained recently by the same
authors in \cite{Laiss}. An explicit formula for $S_{p,\left(  a,d\right)
}^{\left(  r\right)  }\left(  n\right)  $ is given in the following Theorem.

\begin{theorem}
	\label{ful}The hyper-sums of powers of integers $S_{p,\left(  a,d\right)
	}^{\left(  r\right)  }\left(  n\right)  $ is given by%
	\[
	S_{p,\left(  a,d\right)  }^{\left(  r\right)  }\left(  n\right)
	={\sum\limits_{i=0}^{n}}\dbinom{n+r-i}{r}\left(  a+id\right)  ^{p}.
	\]
	
\end{theorem}

\begin{proof}
	These facts are easily verified by induction on $r$ with%
	\[%
	{\displaystyle\sum\limits_{j=i}^{n}}
	\binom{j-i+r-1}{r-1}=\binom{n+r-i}{r}.
	\]
	
\end{proof}

We will now derive a few further consequences of Theorem \ref{ful}.

\begin{corollary}
	The exponential generating function of the hyper-sums of powers of integers
	$S_{p,\left(  a,d\right)  }^{\left(  r\right)  }\left(  n\right)  $ is given
	by%
	\begin{equation}
	\bigskip{\sum\limits_{p\geq0}}S_{p,\left(  a,d\right)  }^{\left(  r\right)
	}\left(  n\right)  \frac{z^{p}}{p!}={\sum\limits_{k=0}^{n}}\dbinom{n+r-k}%
	{r}e^{\left(  a+kd\right)  z}. \label{expGen}%
	\end{equation}
	
\end{corollary}

\begin{proof}
	We have%
	\begin{align*}
	{\sum\limits_{p\geq0}}S_{p,\left(  a,d\right)  }^{\left(  r\right)  }\left(
	n\right)  \frac{z^{p}}{p!}  &  ={\sum\limits_{p\geq0}}\left(  {\sum
		\limits_{k=0}^{n}}\dbinom{n+r-k}{r}\left(  a+kd\right)  ^{p}\right)
	\frac{z^{p}}{p!}\\
	&  =\sum\limits_{k=0}^{n}\dbinom{n+r-k}{r}{\sum\limits_{p\geq0}}\frac{\left(
		\left(  a+kd\right)  z\right)  ^{p}}{p!}\\
	&  ={\sum\limits_{k=0}^{n}}\dbinom{n+r-k}{r}e^{\left(  a+kd\right)  z}%
	\end{align*}
	
\end{proof}

\begin{theorem}
	The exponential generating function of the hyper-sums of powers of integers
	$S_{p,\left(  a,d\right)  }^{\left(  r\right)  }\left(  n\right)  $ is%
	\begin{equation}%
		{\displaystyle\sum\limits_{p\geq0}}
		S_{p,\left(  a,d\right)  }^{\left(  r\right)  }\left(  n\right)  \frac{z^{p}%
	}{p!}=\dbinom{n+r+1}{r+1}e^{az}\text{ }_{2}F_{1}\left(
	\begin{array}
	[c]{c}%
	1,-n\\
	r+2
	\end{array}
	;1-e^{dz}\right)  , \label{Res1}%
	\end{equation}
	where $_{2}F_{1}\left(
	\begin{array}
	[c]{c}%
	a,b\\
	c
	\end{array}
	;z\right)  $ denotes the Gaussian hypergeometric function defined by%
	\[
	{\sum\limits_{n\geq0}}\frac{\left(  a\right)  ^{\overline{n}}\left(  b\right)
		^{\overline{n}}}{\left(  c\right)  ^{\overline{n}}}\frac{z^{n}}{n!},
	\]
and $\left(  x\right)  ^{\overline{n}}$ denotes the Pochhammer symbol defined by
	\[
	\left(  x\right)  ^{\overline{0}}=1\ \ and \ \ \left( x\right)
		^{\overline{n}}=x(x+1)\cdots(x+n-1).
	\]
\end{theorem}

\begin{proof}
	From $\left(  \ref{expGen}\right)$, we have
	\begin{align*}
	{\sum\limits_{p\geq0}}S_{p,\left(  a,d\right)  }^{\left(  r\right)  }\left(
	n\right)  \frac{z^{p}}{p!}  &  =e^{az}%
		{\displaystyle\sum\limits_{k=0}^{n}}
		\dbinom{k+r}{r}e^{d\left(  n-k\right)  z}\\
	&  =\frac{\left(  n+r+1\right)  !e^{az}}{n!r!}%
		{\displaystyle\sum\limits_{k=0}^{n}}
		\dbinom{n}{k}\frac{(n-k)!\left(  k+r\right)  !}{\left(  n+r+1\right)
		!}e^{d\left(  n-k\right)  z}\\
	&  =\dbinom{n+r+1}{r+1}\left(  r+1\right)  e^{az}%
	{\displaystyle\sum\limits_{k=0}^{n}}
		\dbinom{n}{k}e^{d\left(  n-k\right)  z}%
		{\displaystyle\int\limits_{0}^{1}}
		\left(  1-x\right)  ^{r+k}x^{n-k}dx\\
	&  =\dbinom{n+r+1}{r+1}\left(  r+1\right)  e^{az}%
		{\displaystyle\int\limits_{0}^{1}}
		\left(  1-x\right)  ^{r}\left(
		{\displaystyle\sum\limits_{k=0}^{n}}
	\dbinom{n}{k}\left(  xe^{dz}\right)  ^{\left(  n-k\right)  }\left(
	1-x\right)  ^{k}\right)  dx\\
	&  =\dbinom{n+r+1}{r+1}e^{az}\left(  r+1\right)
		{\displaystyle\int\limits_{0}^{1}}
		\left(  1-x\right)  ^{r}\left(  1-x+xe^{dz}\right)  ^{n}dx.
	\end{align*}
	It follows from the theory of hypergeometric functions that the Gaussian
	hypergeometric function $_{2}F_{1}\left(
	\begin{array}
	[c]{c}%
	1,-n\\
	r+2
	\end{array}
	;1-e^{dz}\right)  $ has an integral representation given by%
	\[
	\text{ }_{2}F_{1}\left(
	\begin{array}
	[c]{c}%
	1,-n\\
	r+2
	\end{array}
	;1-e^{dz}\right)  =\left(  r+1\right)  {\int\limits_{0}^{1}}\left(
	1-x\right)  ^{r}\left(  1-x+xe^{dz}\right)  ^{n}dx.
	\]
	which implies $\left(  \ref{Res1}\right)  .$
\end{proof}

\begin{theorem}
	The ordinary generating function of the hyper-sums of powers of integers
	$S_{p,\left(  a,d\right)  }^{\left(  r\right)  }\left(  n\right)  $ is given
	by%
	\begin{equation}
	{\sum\limits_{r\geq0}}S_{p,\left(  a,d\right)  }^{\left(  r\right)  }\left(
	n\right)  z^{r}=\frac{1}{\left(  1-z\right)  ^{n+1}}{\sum\limits_{i=0}^{n}%
	}\left(  1-z\right)  ^{i}\left(  a+id\right)  ^{p} \label{expr}%
	\end{equation}
	
\end{theorem}

\begin{proof}
	Since
	\[
	{\sum\limits_{r\geq0}}\dbinom{n+r-i}{r}{z}^{r}=\left(  1-z\right)  ^{i-n-1},
	\]
	which implies $\left(  \ref{expr}\right)  .$
\end{proof}

\begin{theorem}
	The double generating function of the hyper-sums of powers of integers
	$S_{p,\left(  a,d\right)  }^{\left(  r\right)  }\left(  n\right)  $ is given
	by%
	\[%
		{\displaystyle\sum\limits_{r\geq0}}
		{\displaystyle\sum\limits_{p\geq0}}
		S_{p,\left(  a,d\right)  }^{\left(  r\right)  }\left(  n\right)  \frac{z^{p}%
	}{p!}t^{r}=\frac{e^{az}-\left(  1-t\right)  ^{n+1}e^{\left(  a+\left(
			n+1\right)  d\right)  z}}{\left(  1-t\right)  ^{n+1}\left(  1-\left(
		1-t\right)  e^{dz}\right)  }.
	\]
	
\end{theorem}

\begin{proof}
	From $\left(  \ref{expGen}\right)  $ and $\left(  \ref{expr}\right)$, we obtain

\begin{align*}%
{\displaystyle\sum\limits_{r\geq0}}
{\displaystyle\sum\limits_{p\geq0}}
S_{p,\left(  a,d\right)  }^{\left(  r\right)  }\left(  n\right)  \frac{z^{p}%
}{p!}t^{r}  &  =%
{\displaystyle\sum\limits_{s=0}^{n}}
\left(
{\displaystyle\sum\limits_{r\geq0}}
\dbinom{n+r-s}{r}t^{r}\right)  e^{\left(  a+sd\right)  z}\\
&  =\frac{e^{az}}{\left(  1-t\right)  ^{n+1}}%
{\displaystyle\sum\limits_{s=0}^{n}}
\left(  \left(  1-t\right)  e^{dz}\right)  ^{s}\\
&  =\frac{e^{az}}{\left(  1-t\right)  ^{n+1}}\left[  \frac{1-\left(
	1-t\right)  ^{n+1}e^{\left(  n+1\right)  dz}}{1-\left(  1-t\right)  e^{dz}%
}\right] \\
&  =\frac{e^{az}-\left(  1-t\right)  ^{n+1}e^{\left(  a+\left(  n+1\right)
		d\right)  z}}{\left(  1-t\right)  ^{n+1}\left(  1-\left(  1-t\right)
	e^{dz}\right)  }.
\end{align*}
\end{proof}

Now, according to the well-known formula, for $n\in\mathbb{N}\text{ and }%
m\in\mathbb{N}^{\ast}$
\[
_{2}F_{1}\left(
\begin{array}
[c]{c}%
-n,1\\
m
\end{array}
;z\right)  =\frac{n!\left(  z-1\right)  ^{m-2}}{\left(  m\right)
	^{\overline{n}}z^{m-1}}\left[  \sum_{k=0}^{m-2}\frac{\left(  n+1\right)
	^{\overline{k}}}{k!}\left(  \frac{z}{z-1}\right)  ^{k}-\left(  1-z\right)
^{n+1}\right]  .
\]
we can rewrite the exponential generating function of the hyper-sums of powers
of integers $S_{p}^{\left(  r\right)  }\left(  n\right)  $ as
\begin{equation}
{\sum\limits_{p\geq0}}S_{p,\left(  a,d\right)  }^{\left(  r\right)  }%
\frac{z^{p}}{p!}=\frac{e^{\left(  a+d\left(  r+\left(  n+1\right)  \right)
		\right)  z}}{\left(  e^{dz}-1\right)  ^{r+1}}-%
{\displaystyle\sum\limits_{k=0}^{r}}
\dbinom{n+k}{k}\frac{e^{\left(  a+\left(  r-k\right)  d\right)  z}}{\left(
	e^{dz}-1\right)  ^{r-k+1}}. \label{12345}%
\end{equation}

The next result gives an explicit formula for $S_{p,\left(  a,d\right)
}^{\left(  r\right)  }\left(  n\right)  $ involving the generalized Bernoulli
polynomials. Recall that the generalized Bernoulli polynomials $B_{n}^{\left(  \alpha\right)  }\left(  x\right)  $ of degree $n$ in $x$ are
defined by the exponential generating function 

\begin{equation}
\left(  \frac{z}{e^{z}-1}\right)  ^{\alpha}e^{xz}={\sum\limits_{n\geq0}%
}B_{n}^{\left(  \alpha\right)  }\left(  x\right)  \frac{z^{n}}{n!}
\label{GenBer}%
\end{equation}
for arbitrary parameter $\alpha$. In particular, $B_{n}^{\left(  1\right)  }\left(  x\right)
:=B_{n}\left(  x\right)  $ denotes the classical Bernoulli
polynomials with $B_{1}\left(  0\right)  =-\frac{1}{2}$. For a recent treatment see \cite{Boutiche1,Srivastava}

\begin{theorem}
	\label{nex}For all $n,p,r\geq0,$ we have%
	\begin{multline*}
	S_{p,\left(  a,d\right)  }^{\left(  r\right)  }\left(  n\right)
	=\frac{p!d^{p}}{\left(  p+r+1\right)  !}B_{p+r+1}^{^{\left(  r+1\right)  }%
	}\left(  \frac{a}{d}+\left(  r+\left(  n+1\right)  \right)  \right) \\
	-p!d^{p}%
	{\displaystyle\sum\limits_{k=0}^{r}}
	\dbinom{n+k}{k}\frac{1}{\left(  p+r+1-k\right)  !}B_{p+r+1-k}^{^{\left(
			r-k+1\right)  }}\left(  \frac{a}{d}+\left(  r-k\right)  \right)
	\end{multline*}
	
\end{theorem}

\begin{proof}
	By (\ref{12345}) and (\ref{GenBer}) we have%
	\begin{multline}%
	{\displaystyle\sum\limits_{p\geq0}}
	S_{p,\left(  a,d\right)  }^{\left(  r\right)  }\left(  n\right)  \frac{z^{p}%
	}{p!}=-%
	{\displaystyle\sum\limits_{k=0}^{r}}
	\dbinom{n+k}{k}%
	{\displaystyle\sum\limits_{p\geq0}}
	d^{p-r+k-1}B_{p}^{^{\left(  r-k+1\right)  }}\left(  \frac{a}{d}+\left(
	r-k\right)  \right)  \frac{z^{p-r+k-1}}{p!}\\
	+%
	{\displaystyle\sum\limits_{p\geq0}}
	d^{p-r-1}B_{p}^{^{\left(  r+1\right)  }}\left(  \frac{a}{d}+\left(  r+\left(
	n+1\right)  \right)  \right)  \frac{z^{p-r-1}}{p!}\nonumber
	\end{multline}
	After some rearrangement, we find%
	\begin{multline*}
	{\sum\limits_{p\geq0}}S_{p,\left(  a,d\right)  }^{\left(  r\right)  }\left(
	n\right)  \frac{z^{p}}{p!}={\sum\limits_{p\geq0}}\frac{z^{p}}{p!}\left(
	\frac{p!d^{p}}{\left(  p+r+1\right)  !}B_{p+r+1}^{^{\left(  r+1\right)  }%
	}\left(  \frac{a}{d}+\left(  r+\left(  n+1\right)  \right)  \right)  \right.
	\\
	-\left.  p!d^{p}%
	{\displaystyle\sum\limits_{k=0}^{r}}
	\dbinom{n+k}{k}\frac{1}{\left(  p+r+1-k\right)  !}B_{p+r+1-k}^{^{\left(
			r-k+1\right)  }}\left(  \frac{a}{d}+\left(  r-k\right)  \right)  \right)
	\end{multline*}
	Equating the coefficient of $\frac{z^{p}}{p!},$ we get the result.
\end{proof}

When $r=0$, Theorem \ref{nex} reduces to (\ref{BBr}).


\end{document}